\documentclass{amsart}

\usepackage[nobysame]{amsrefs}
\usepackage{amsthm}
\usepackage{amssymb}

\usepackage{tikz-cd}
\usepackage{colonequals}
\usepackage{enumerate}
\usepackage{eucal}

\usepackage{hyperref}
\hypersetup{%
  bookmarksnumbered=true,%
  colorlinks=true,%
  linkcolor=blue,%
  citecolor=blue,%
  filecolor=blue,%
  menucolor=blue,%
  urlcolor=blue,%
  bookmarksopen=true,%
  bookmarksdepth=2,%
  pageanchor=true}

\makeatletter
\@namedef{subjclassname@2020}{\textup{2020} Mathematics Subject Classification}
\makeatother

\numberwithin{equation}{section}

\theoremstyle{plain}
\newtheorem{theorem}[equation]{Theorem}
\newtheorem{proposition}[equation]{Proposition}
\newtheorem{lemma}[equation]{Lemma} 
\newtheorem{corollary}[equation]{Corollary}

\theoremstyle{definition}
\newtheorem{definition}[equation]{Definition}
\newtheorem{chunk}[equation]{}

\theoremstyle{remark}
\newtheorem{remark}[equation]{Remark}

\newcommand{\acat}{\operatorname{C}_{\mathcal{O}}}
\newcommand{\cat}[1]{\mathsf{#1}}
\newcommand{\cmod}[1]{\Psi_{#1}}
\newcommand{\coker}{\operatorname{coker}}
\newcommand{\con}[1]{\Phi_{#1}}
\newcommand{\depth}{\operatorname{depth}}
\newcommand{\ecoh}{\operatorname{F}}
\newcommand{\eps}{\varepsilon}
\newcommand{\ext}{\operatorname{Ext}}
\newcommand{\height}{\operatorname{height}}
\newcommand{\Hom}{\operatorname{Hom}}
\newcommand{\idmap}{\operatorname{id}}
\newcommand{\Ker}{\operatorname{Ker}}
\newcommand{\length}{\operatorname{length}}
\newcommand{\ord}{\operatorname{ord}}
\newcommand{\pos}[1]{[\![{#1}]\!]}
\newcommand{\rank}{\operatorname{rank}}
\newcommand{\tors}{\operatorname{tors}}
\newcommand{\tfree}[1]{{#1}^{\operatorname{tf}}}
\newcommand{\vf}{\varphi}
\newcommand{\cL}{\mathcal{L}}
\newcommand{\tcL}{\tilde{\mathcal{L}}}
\newcommand{\full}{\mathrm{full}}
\newcommand{\Sel}{\mathrm{Sel}}

\newcommand{\mco}{\mathcal{O}}
\newcommand{\CF}{\mathcal{F}}
\newcommand{\bs}{\boldsymbol}

\newcommand{\fm}{\mathfrak{m}} 
\newcommand{\fp}{\mathfrak{p}}

\newcommand{\bbQ}{\mathbb{Q}}
\newcommand{\bbZ}{\mathbb{Z}}
\newcommand{\bbT}{\mathbb{T}}
\newcommand{\bbA}{\mathbb{A}}

\newcommand{\ad}{\operatorname{ad}}
\newcommand{\GL}{\operatorname{GL}}
\newcommand{\Gal}{\operatorname{Gal}}

\DeclareFontFamily{U}{wncy}{}
    \DeclareFontShape{U}{wncy}{m}{n}{<->wncyr10}{}
    \DeclareSymbolFont{mcy}{U}{wncy}{m}{n}
    \DeclareMathSymbol{\Sha}{\mathord}{mcy}{"58} 
    
\begin{document}

\title[Congruence modules and Galois cohomology]{Congruence modules in higher codimension and zeta lines in Galois cohomology}

\author[S.~B.~Iyengar]{Srikanth B.~Iyengar}
\address{Department of Mathematics,
University of Utah, Salt Lake City, UT 84112, U.S.A.}
\email{iyengar@math.utah.edu}

\author[C.~B.~Khare]{Chandrashekhar  B. Khare}
\address{Department of Mathematics,
University of California, Los Angeles, CA 90095, U.S.A.}
\email{shekhar@math.ucla.edu}

\author[J.~Manning]{Jeffrey Manning}
\address{Mathematics Department,
Imperial College, London, SW7 2RH, UK}
\email{jeffamanning@gmail.com}

\author[E.~Urban]{Eric Urban}
\address{Mathematics Department,
Columbia University, New York, U. S. A.}
\email{urban@math.columbia.edu}

\date{\today}


\begin{abstract} 
This work builds on that in \cite{Iyengar/Khare/Manning:2022a} where a notion of congruence modules  in higher codimension is introduced. The main new results are a criterion for detecting regularity of local rings in terms of congruence modules, and a more refined version of a result tracking the change of congruence modules under deformation is proved. Number theoretic applications include the construction of  canonical lines in certain  Galois cohomology groups arising from adjoint motives of Hilbert modular forms. 
\end{abstract}

\maketitle

\section{Introduction}
\label{se:intro}
Let $p$ be a prime number, $\mco$  the ring of integers of a finite extension of $\bbQ_p$, and $R$ a  complete noetherian local $\mco$-algebras with an $\mco$-algebra morphism $\lambda\colon R \to \mco$ such that the local ring $R_\fp$, where $\fp\colonequals\Ker(\lambda)$, is regular.

 In \cite{Iyengar/Khare/Manning:2022a} we develop an analog of the Wiles-Lenstra-Diamond  numerical  criterion in arbitrary  codimension, with the original criterion  (see \cite{Diamond:1997,Wiles:1995}) corresponding to the codimension 0 case. This gives a criterion for  finitely generated $R$-module $M$ to have a free direct summand and for $R$ to be a complete intersection ring. This involves two invariants associated to  $\lambda$: the torsion $\Phi_\lambda(R)$  of the cotangent space of $\lambda$,  and the congruence module $\Psi_\lambda(M)$. The codimension  $c$ of the augmentation $\lambda$ is the height of $\fp$, or equivalently under our regularity assumption, the dimension of $R_\fp$.  We denote  the category of such pairs $(R,\lambda)$ by $\acat(c)$.  The earlier work of \cite{Diamond:1997} corresponds to the case $c=0$ in which case $\fp/\fp^2=\Phi_\lambda(R)$.
 
Wiles uses the criterion in his work on the modularity of elliptic curves over $\bbQ$ to go from modularity lifting theorems in the minimal case to those in the non-minimal case. In  \cite{Iyengar/Khare/Manning:2022a} the numerical criterion in higher codimension is used along the same lines to prove integral modularity lifting results  for non-minimal lifts  in situations of positive defect which arise in considering Galois representations  over imaginary quadratic fields. 

 In this work we explore the meaning of   the invariants $\Phi_\lambda(R)$  and $\Psi_\lambda(M)$ in  certain number theoretic situations, relating them to the index of zeta elements in global Galois cohomology groups.  
 
 To begin with we  focus  on one of the key ideas of \cite{Iyengar/Khare/Manning:2022a}, namely the definition and properties of  congruence modules, and congruence ideals, associated to an augmentation $\lambda$ in arbitrary codimension. This is the content of Section~\ref{se:ca-section}. The highlights are a characterisation, Theorem \ref{th:regular}, of regularity of rings  $(A,\lambda)$ in our category $\acat(c)$ in terms of vanishing of the invariants $\Phi_\lambda(A)$ and $\Psi_\lambda(A)$); a more transparent description, in \S \ref{sec:pairing}, of the connection between duality and our congruence modules than in the earlier paper, and a refinement,  Theorem \ref{th:deformation-with-lambda}, of a result about deformation invariance of Wiles defect.
  
 Section~\ref{se:cotorsion} focuses on number theoretic applications, and concerns the congruence ideal attached to an augmentation of Hida’s ordinary Hecke algebra $\bbT^{\rm ord}$, or ordinary deformation ring $R^{\rm ord}$,  arising from  cohomological  Hilbert modular forms  $f$ over totally real fields $F$. The functorial properties of the congruence ideal are used to relate it to lines (that is to say, free $\mco$-modules of rank one) in the Galois cohomology  with coefficients in $\ad \rho_f$; see~Theorem \ref{th:zeta}. The  index of their image under global-to-local restriction  maps to singular local Galois cohomology $H^1_{{\rm ord}/f}(G_p,\ad \rho_f)$ is related to classical congruence module of $f$, and to special values of adjoint $L$-function 
 \[
 L^{\rm alg}(1,\ad \rho_f)\,,
 \] 
 by work of Hida \cite{Hida:1981} and its generalization by Dimitrov \cite{Dimitrov:2009}. This is connected to  the ``zeta elements’’  of \cite[Theorem 1.1]{Urban:2021}. The terminology is due to Kato~\cite{Kato:2004}, who used it for  elements of Galois cohomology  he constructs in a related context, arising from Beilinson-Flach classes.  In \cite{Urban:2021} exact sequences of Selmer groups are used while here we use in addition our congruence modules in higher codimension, allowing one to eliminate some hypotheses which arise from relying  on $R=\bbT$  theorems. 

 This connection arises  from the following circumstance.  Let $\bs{t}\colonequals t_1,\dots,t_c$ be indeterminates, $\Lambda_c\colonequals \mco\pos{\bs{t}}$ the power-series ring, and $\Lambda_c\to \mco$ the natural augmentation. Fix $(A,\lambda)$ in $\acat(c)$ equipped with a finite flat map $\iota\colon\Lambda_c\to A$ of $\mco$-algebras over $\mco$, so that composite map 
\[
\Lambda_c\xrightarrow{\ \iota\ } A\xrightarrow{\ \lambda\ }\mco\,,
\]
is the augmentation. For $A_0=A/(\bs{t})$ one has a map
\[
\ext^c_{A}(\mco,A)(=A_0[\ker \lambda])  \to \bigwedge^c_\mco \Hom(\fp/\fp^2,\mco)\,,
\]
whose cokernel is  $\cmod{\lambda}(A)$. Abstractly both domain and range are simply $\mco$’s.  In applications when $A$ is a nearly ordinary  deformation ring and $\lambda$ arises from the classical form $f$, the range  is a Selmer group  as  $\Hom(\fp/\fp^2,\mco)=H^1_{\cL}(G_{F,S},\ad \rho_f)$, with local conditions  $\cL=(\cL_v)$ such that $\cL_v \subset H^1(G_v,\ad \rho_f)$ for $v \in S$  and $\cL_v$ is the unramified subspace  $H^1_{\rm unr}(G_v,\ad \rho_f)$   for $v$ not in $S$.     When $A$ is a nearly ordinary Hecke algebra $  \Hom(\fp/\fp^2,\mco)$ is a subspace of $H^1(G_{F,S},\ad \rho_f)$. This gives a ``pure thought’’  construction of  canonical  lines in Galois cohomology; see Theorem~\ref{th:zeta}.

 \section{Higher congruence modules and Wiles defects}
 \label{se:ca-section}
We being by recalling the setup of \cite{Iyengar/Khare/Manning:2022a}. This section complements the material presented in \cite[Part~1]{Iyengar/Khare/Manning:2022a}, where the commutative algebraic  aspects of the theory of congruence modules is developed. There are some new results, the main ones being Theorems~\ref{th:regular} and \ref{th:deformation-with-lambda}, and  Proposition \ref{pr:iso-criteria}. Along the way we provide also a different perspective and new proofs of some key results from \cite{Iyengar/Khare/Manning:2022a}.

\begin{chunk}
Let $\mco$ be a complete discrete valuation ring, with valuation $\ord(-)$ and uniformizer $\varpi$. Throughout we fix a complete local $\mco$-algebra $A$ and a finitely generated $A$-module $M$. Given a map $\lambda\colon A\to \mco$ of $\mco$-algebras, set
\[
\fp_\lambda \colonequals \Ker\lambda\quad\text{and}\quad c\colonequals \height{\fp_{\lambda}}\,.
\]
For any finitely generated $A$-module $M$, set
	\[
	\ecoh^i_{\lambda}(M)\colonequals \tfree{\ext^i_A(\mco,M)}
	\]
	the torsion-free quotient of the $\mco$-module  $\ext^i_A(\mco,M)$. Here $\mco$ is viewed as an $A$-module via $\lambda$. The \emph{congruence module} of $M$ at $\lambda$ is the $\mco$-module
	\[
	\cmod{\lambda}(M)\colonequals \coker\left(\ecoh^c_{\lambda}(M) \xrightarrow{\ \ecoh^c_{\lambda}(\lambda\otimes M)\ }\ecoh^c_{\lambda}(M/\fp_\lambda M)\right)\,.
	\]
We have also to consider $\mco$-module
	\[
	\con{\lambda}(A)\colonequals \tors(\fp_\lambda/\fp_\lambda^2)\,,
	\]
	namely, the torsion part of the cotangent module $\fp_\lambda/\fp_\lambda^2$ of $\lambda$.
\end{chunk}

We say an $A$-module $M$ has a certain property \emph{at $\lambda$} if the $A_{\fp_\lambda}$-module $M_{\fp_\lambda}$ has the stated property. For instance we say $A$ is regular at $\lambda$ to mean that the local ring $A_{\fp_\lambda}$ is regular. The starting point of our work is the following result; see
	\cite[Theorem 2.5 and Lemma~2.6]{Iyengar/Khare/Manning:2022a}.

\begin{theorem}
	\label{th:acat}
	With  $\lambda\colon A\to \mco$ as above, the following conditions are equivalent:
	\begin{enumerate}[\quad\rm(1)]
		\item
		The local ring $A$ is regular at $\lambda$.
		\item
		The rank of the $\mco$-module $\fp_\lambda/\fp^2_\lambda$ is $\height \fp_\lambda$.
		\item
		The $\mco$-module $\cmod{\lambda}(A)$ is torsion.
		\item
		The $\mco$-module $\cmod{\lambda}(M)$ is torsion for each finitely generated $A$-module $M$.
	\end{enumerate}
	Moreover, when these conditions hold the $\mco$-module $\cmod{\lambda}(A)$ is cyclic. \qed
\end{theorem}

	Condition (2) is that the embedding dimension of the ring $A_{\fp_\lambda}$ equals its Krull dimension, so (1)$\Leftrightarrow$(2) is one definition of regularity; see \cite[Definition~2.2.1]{Bruns/Herzog:1998}. The key input in proving (1)$\Leftrightarrow$(3) is the following result due to Lescot~\cite{Lescot:1983}; see also \cite{Avramov/Iyengar:2013}.

 \begin{chunk}
 \label{ch:lescot}
    A noetherian local ring $R$ is regular if and only if the map
    \[
    \ext_R(k,R)\longrightarrow \ext_R(k,k)
    \]
    induced by the canonical surjection $R\to k$ to the residue field of $R$, is nonzero. When this is the case, the map above is nonzero in (upper) degree $\dim R$.
\end{chunk}

The result below is implicit in the proof of (1)$\Rightarrow$(4) in Theorem~\ref{th:acat}, in \cite{Iyengar/Khare/Manning:2022a}. We make it explicit, for it is used also in proving Lemma~\ref{le:injective-at-lambda} and Theorem~\ref{th:regular} below.

\begin{lemma}
\label{le:iso-at-lambda}
Let $\eps\colon R\to S$ be a surjective map of noetherian rings such that the ideal $\Ker(\eps)$ is generated by a regular sequence, and set $c\colonequals \dim R-\dim S$. For any $R$-module $M$ the map below is bijective:
\[
\ext^c_R(S,\eps\otimes_RM)\colon \ext^c_R(S,M)\longrightarrow \ext^c_R(S,S\otimes_RM)\,.
\]
\end{lemma}

\begin{proof}
The Koszul complex, say $K$,  on any regular sequence generating $\Ker(\eps)$ is a minimal resolution of $S$ as an $R$-module. The map $\ext^c_R(S,\eps\otimes_RM)$ is the one obtained in cohomology in (upper) degree $c$ from the morphism  in the upper row of the following commutative diagram of complexes:
\[
\begin{tikzcd}[column sep=large]
\Hom_R(K,M) \ar[rr,"{\Hom_R(K,\eps\otimes M)}"] && \Hom_R(K,S\otimes_R M) \\
\Hom_R(K,R)\otimes_RM \ar[u,"\cong"] \ar[rr, "{\Hom_R(K,R)\otimes (\eps\otimes M)}"] && \Hom_R(K,R)\otimes_R (S\otimes_RM) \ar[u,"\cong" swap]
\end{tikzcd}
\]
The vertical maps are isomorphisms because $K$ is a finite free $R$-complex.  It is clear that the map in the lower row induces  a bijection in cohomology in the top degree, $c$. Thus the same holds for the one in the upper row, as claimed. 
\end{proof}

We denote $\acat$ the category whose objects are pairs $(A,\lambda)$ satisfying the equivalent conditions in Theorem~\ref{th:acat}. A morphism $\vf\colon (A,\lambda)\to (A',\lambda')$ in this category is a map of $\mco$-algebras $\vf\colon A\to A'$ over $\mco$; that is to say, with $\lambda'\circ \vf = \lambda$. We write $\acat(c)$ for the subcategory of $\acat$ consists of pairs $(A,\lambda)$ such that $\height \fp_\lambda=c$.

\begin{lemma}
\label{le:injective-at-lambda}
For any $(A,\lambda)$ in $\acat(c)$ and finitely generated $A$-module $M$ the map
\[
\ecoh^c_{\lambda}(\lambda\otimes_AM)\colon \ecoh^c_{\lambda}(M) \longrightarrow \ecoh^c_{\lambda}(M/\fp_\lambda M)
\]
is one-to-one.
\end{lemma}

\begin{proof}
Set $R\colonequals A_{\fp_\lambda}$ and let $\eps \colon R\to E$ be the map obtained by localizing $\lambda$ at $\fp_\lambda$; here $E$ is the residue field of $R$, which is also the field of fractions of $\mco$. Since injectivity of a map of torsion-free $\mco$-modules can be detected after passing to the field of fractions, it suffices to check that the map 
\[
\ecoh^c_\lambda(\lambda\otimes_AM)_{\fp_\lambda} \cong \ext_R^c(E,\eps\otimes_R M)
\]
is one-to-one. Since $R$ is regular the ideal $\Ker(\eps)$ is generated by a regular sequence of length $c$; see~\cite[Proposition~2.2.4]{Bruns/Herzog:1998}.  It remains to apply Lemma~\ref{le:iso-at-lambda}. 
\end{proof}

In the work of Hida~\cite{Hida:1981} and Ribet~\cite{Ribet:1983} congruence modules (for codimension $c=0$)  are attached to augmentations $\bbT\to \mco$ of Hecke algebras $\bbT$  that act faithfully on certain localized Betti cohomology groups $H^1(X_1(N),\mco)_\fm$. They  measure the complexity of $\bbT$ and their vanishing is equivalent to $\bbT$ being smooth, namely just $\mco$. Analogously we show in the result below  that for rings $A \in \acat$  the vanishing of  either the  congruence module $\cmod{\lambda}(A)$ or $\con{\lambda}(A)$, the torsion part of the cotangent module, at any augmentation   $\lambda\colon A \to \mco$  implies $A$ is  smooth.

Unlike most results in  \cite[Part~1]{Iyengar/Khare/Manning:2022a}, the following theorem does not make assumptions on the depth of the ring $A$.

\begin{theorem}
\label{th:regular}
For $(A,\lambda)$ in $\acat$, the local ring $A$ is regular if and only if $\con{\lambda}(A)=0$, if and only if $\cmod{\lambda}(A)=0$.
\end{theorem}
    
\begin{proof}
We first verify that $A$ is regular if and only if $\con{\lambda}(A)=0$. As $A$ is a complete $\mco$-algebra, one has $A\cong P/I$ where $P\colonequals \mco\pos{t_1,\dots,t_n}$, a ring of formal power series over $\mco$, the ideal $I\subseteq (\varpi)(\bs  t)+(\bs  t)^2$, and
$\lambda\colon A\to \mco$ is quotient by $(\bs t)$. 

Let $\bs  f \colonequals f_1,\dots,f_m$ be a minimal generating set for $I$. The cotangent module $\fp_\lambda /\fp_\lambda^2$ depends only on $n$ and the linear part of the $f_i$, in the following sense: By our assumption on $I$,  each $f_i$ has an unique expression of the form
\begin{equation}
\label{eq:A-presentation}
f_i \colonequals \sum_{j=1}^n{u_{ij}}t_j + g_i \qquad\text{with  $u_{ij}\in (\varpi) \mco $ and $g_i\in (\bs  t)^2$.}
\end{equation}
Then one has a presentation
\[
\mco^m \xrightarrow{\ (u_{ij}) \ } \mco^n \longrightarrow \fp_\lambda/\fp_\lambda^2\longrightarrow 0\,.
\]
The torsion part of $\fp_\lambda/\fp_\lambda^2$ is zero if and only if $(u_{ij})=0$, that is to say, $(\bs  f)\subseteq (\bs  t)^2$. Since $A$ is regular at $\lambda$ this condition is equivalent to $\bs  f=0$, as desired.

Next we verify the claim that $A$ is regular if and only if $\cmod{\lambda}(A)=0$.

When $A$ is regular, $\Ker(\lambda\colon A\to\mco)$ is generated by a regular sequence of length $c\colonequals \height (\Ker\lambda)$; see~\cite[Proposition~2.2.4]{Bruns/Herzog:1998}. Thus Lemma~\ref{le:iso-at-lambda} yields that the map $\ext^c_A(\mco,\lambda)$ is one-to-one so $\cmod{\lambda}(A)=0$.

Assume $\cmod{\lambda}(A)=0$.  To verify that $A$ is regular it suffices to verify that the map 
\[
\ext_A(k,\eps)\colon \ext_A(k,A)\to \ext_A(k,k)\,,
\]
induced by the quotient map $\eps\colon A\to k$, is non-zero, for then Lescot's result~\ref{ch:lescot}. 

Let $M$ be a finitely generated $A$-module. The exact sequence 
\begin{equation}
    \label{eq:residue}
    0\longrightarrow \mco \xrightarrow{\ \varpi\ } \mco\longrightarrow k\longrightarrow 0
\end{equation} 
of $A$-modules induces exact sequences of $k$-modules
\[
0 \longrightarrow k\otimes_\mco \ext_A^i(\mco,M)\xrightarrow{\ \eth^{i+1}(M)\ } \ext_A^i(k,M)  \longrightarrow \ext_A^{i+1}(\mco,M)[\varpi]\longrightarrow 0    
\]
For what follows the relevant point is that the maps $\eth^i(M)$ are inclusions.  Set $c\colonequals \height(\fp_\lambda)$ and consider the following commutative diagram of $k$-vector spaces:
\[
\begin{tikzcd}
	k\otimes_\mco \ext^c_A(\mco,A)\arrow{r}\arrow[d,hookrightarrow, "\eth(A)" swap]\arrow[dr,dashrightarrow]
                    & k\otimes_\mco \ext^c_A(\mco,\mco) \arrow[d, hookrightarrow, "\eth(\mco)"] \\
	\ext^{c+1}_A(k,A)\arrow{r}\arrow[rr, bend right, "{\ext^{c+1}_A(k,\eps)}"]
                    & \ext^{c+1}_A(k,\mco) \arrow[r, hookrightarrow]
             &\ext^{c+1}_A(k,k)
	\end{tikzcd}
\]
The map in the top row is induced by $\lambda\colon A\to \mco$ and the ones in the lower row are induced by $A\to \mco\to k$. That the map in the lower right is one-to-one follows by considering the long exact sequence in cohomology that arises by applying $\Hom_A(k,-)$ to the exact sequence \eqref{eq:residue}.
It is easy to verify that the hypothesis $\cmod{\lambda}(A)=0$ implies the map in the top row is nonzero, and hence so is the diagonal map. It then follows from the commutative diagram that the map $\ext^{c+1}_A(k,\eps)$ is nonzero. This is as desired.
\end{proof}

\begin{remark}
Consider the  ordinary  Hida Hecke algebra  $\bbT^{\ord}$ of tame level $N$.   It is finite flat over $\Lambda\colonequals \mco\pos t$, with $t$ the weight variable,  and $\bbT^{\ord}/(t)=\bbT$ is a classical Hecke algebra, acting faithfully on $H^1(X_1(Np),\mco)_\fm$.  Consider an augmentation $\lambda\colon\bbT^{\ord} \to \bbT \to \mco$ arising from a ($p$-stabilized) newform $f \in S_1(\Gamma_1(Np^r))$. The vanishing of the congruence module of $\bbT$ for  the augmentation $\bbT \to \mco$ implies  $\bbT=\mco$ and $\bbT^{\ord}=\Lambda$ while the vanishing of the congruence module for $\bbT^{\ord} \to \mco$ implies that $\bbT^{\ord}$  is smooth,   while $\bbT$ may not be smooth. In other words  $\cmod{\lambda}(\bbT^{\ord})=0$ implies $\bbT^{\ord}=\mco\pos{x}$, furthermore  $x$ can be taken to be the weight variable $t$  if and only if the  classical congruence module $\cmod{\lambda}(\bbT)=0$ also  vanishes. 
\end{remark}

Next we describe a pairing associated with the definition of congruence modules. This too appears in \cite{Iyengar/Khare/Manning:2022a}, but does not play a major role in the development there. The presentation below is more transparent, and is used to give another perspective on some of the subsequent results.

\begin{chunk}
\label{sec:pairing}
For any finitely generated $A$-module $M$, one has natural isomorphisms
\begin{align*}
\tfree{\ext^c_A(\mco,M/\fp_\lambda M)} 
	&\cong  \tfree{\ext^c_A(\mco,\mco)}\otimes_{\mco} \tfree{(M/\fp_\lambda M)} \\
	&\cong \Hom_{\mco}(\Hom_{\mco}(M,\mco), \tfree{\ext^c_A(\mco,\mco)})\,.
\end{align*}
Thus, the map
\[
\ecoh^c_{\lambda}(\lambda\otimes M)\colon \ext^c_A(\mco,M)\longrightarrow \tfree{\ext^c_A(\mco,M/\fp_\lambda M)}
\]
whose cokernel is the congruence module of $M$, is adjoint to the map
\[
\langle -,-\rangle_M \colon \tfree{\ext^c_A(\mco,M)} \otimes_\mco \Hom_{A}(M,\mco) \longrightarrow \tfree{\ext^c_A(\mco,\mco)}\,.
\]
The \emph{congruence ideal} of $M$, with respect to the augmentation $\lambda$, is the image of this pairing:
\[
\eta_\lambda(M)\colonequals \mathrm{Image}\langle -,-\rangle_M\,.
\]
Localizing at $\fp_\lambda$, it is easy to verify that the free $\mco$-modules $\ecoh^c_{\lambda}(M)$ and $\Hom_A(M,\mco)$ have the same rank and that  $\ecoh^c_{\lambda}(\mco)$ has rank one, so
\[
\length_\mco(\mco/\eta_\lambda(M)) \le \length_\mco \cmod{\lambda}(M) \le \rank_\lambda(M)\cdot \length_\mco(\mco/\eta_\lambda(M))\,.
\]
Here $\rank_\lambda(M)$ denotes the rank of $M$ at $\lambda$, that is to say, the rank of the $A_\lambda$-module $M_{\fp_\lambda}$. In particular, when this rank equals $1$, the length of the congruence module can be computed from the pairing.

The pairing above is induced--by passage to torsion-free quotients--by the natural pairing given by composition of morphisms:
\begin{equation}
    \label{eq:pairing1}
\langle -,-\rangle_M \colon \ext^c_A(\mco,M)\otimes_{\mco} \Hom_A(M,\mco) \longrightarrow \ext^c_A(\mco,\mco)\,.
\end{equation}
 Namely, $\ext^c_A(\mco,M)$ can be realized as $\Hom_{\cat D}(\mco,M[c])$, the morphisms in the derived category of $A$ from $\mco$ to $M[c]$, and given such a morphism $f$ and a map $g\colon M\to \mco$, the pairing above is
\[
\langle f,g\rangle \colonequals g\circ f \colon \mco \to \mco[c]\,.
\]
In terms of the Yoneda interpretation of $\ext^c_A(\mco,M)$ as equivalence classes 
\[
0\longrightarrow M\longrightarrow X_{c-1} \longrightarrow X_{c-1}\longrightarrow 
    \cdots \longrightarrow X_0 \longrightarrow \mco \longrightarrow 0
\]
of exact sequences, the pairing is given by taking push-out along $g\colon M\to \mco$.
\end{chunk}

\subsubsection*{Cohen-Macaulay modules}
When $M$ is Cohen-Macaulay of dimension $c+1$, local duality yields an identification
\[
\Hom_A(\mco,M) \cong \ext^c_A(\mco,M^\vee) \quad \text{where $M^\vee \cong H_{c+1}(\mathrm{RHom}_A(M,\omega_A))$}
\]
Here $\omega_A$ is the dualizing complex of $A$, normalized as in \cite{stacks-project}; see \cite[\S4]{Iyengar/Khare/Manning:2022a} for details. With this identification, the pairing \eqref{eq:pairing1} takes the form
\begin{equation}
\label{eq:pairing2}
\langle -,-\rangle \colon \ext^c_A(\mco,M)\otimes_{\mco} \ext^c_A(\mco, M^{\vee}) 
    \longrightarrow \ext^c_A(\mco,\mco)\,.
\end{equation}
See \cite[Proposition~4.7]{Iyengar/Khare/Manning:2022a}.

The pairing~\eqref{eq:pairing1} can be described concretely when $c\le 1$. One simplification that occurs then is that $\ext^1_A(\mco,\mco)$ is already torsion-free, as is explained below.

Consider the exact sequence 
\begin{equation}
    \label{eq:Opres}
0 \longrightarrow \fp_\lambda \longrightarrow A \longrightarrow \mco \longrightarrow 0\,.    
\end{equation}
Applying $\Hom_A(-,M)$ yields the exact sequence
\[ 
M\cong \Hom_A(A,M) \longrightarrow \Hom_A(\fp_\lambda,M) 
    \longrightarrow \ext^1_A(\mco,M) \longrightarrow \ext_A^1(A,M)=0\,.
    \] 
This justifies the following result.

\begin{lemma}
\label{le:c=1}
\pushQED{\qed}
For any $c\ge 0$, and any $A$-module $M$, there is a natural isomorphism of $\mco$-modules 
\[
\ext^1_A(\mco,M)\cong \coker(M\to \Hom(\fp_\lambda,M))\,. \qedhere
\]
\end{lemma}

The isomorphism above assigns to any $A$-linear map $f\colon \fp_\lambda \to M$ the exact sequence obtained by push-out of the exact sequence in \eqref{eq:Opres} along $f$:
\[
\begin{tikzcd}
    0 \arrow[r] & \fp_\lambda \arrow[d,"f" swap] \arrow[r]  & A \arrow[d] \arrow[r]  & \mco  \arrow[d, equal] \arrow[r]  & 0 \\
    0 \arrow[r] & M \arrow[r]  & X \arrow[r]  & \mco \arrow[r]  & 0 
\end{tikzcd}
\]
The natural map $\mco\to \Hom_{A}(\fp_\lambda,\mco)$ is zero, so for $M\colonequals\mco$ the isomorphism in Lemma~\ref{le:c=1} becomes
\begin{equation}
\label{eq:ecoh1}
\ext^1_A(\mco,\mco)\cong \Hom_A(\fp_\lambda,\mco)\cong \Hom_{\mco}(\fp_\lambda/\fp^2_\lambda,\mco)
\end{equation}
which is already torsion-free. 

Now we return to the pairing \eqref{eq:pairing1}.

\subsection*{The case $c=0$}
Since $\Hom_A(\mco,M)=M[\fp_\lambda]$, the $\fp_\lambda$-torsion submodule of $M$, the pairing \eqref{eq:pairing1} becomes
\begin{gather*}
M[\fp_\lambda]\otimes_\mco \Hom_A(M,\mco) \longrightarrow \mco \\
 m \otimes f \mapsto f(m)
\end{gather*}
When $\depth_AM\ge 1$, one has $M^\vee\cong \Hom_\mco (M,\mco)$ this pairing is equivalent to the one given by the composition
\[
M[\fp_\lambda]\otimes_\mco M^{\vee}[\fp_\lambda] \longrightarrow M\otimes_{\mco} M^{\vee} \longrightarrow \mco
\]
where the map on the right is the obvious one.

\subsection*{The case $c=1$}

With this description, for $c=1$ the pairing \eqref{eq:pairing1} is induced by the obvious pairing
\[
\Hom_A(\fp,M)\otimes_A \Hom_A(M,\mco) \longrightarrow \Hom_A(\fp,\mco)\cong \Hom_{\mco}(\fp/\fp^2,\mco)
\]
given by composition of maps. Since $\ext^1_A(\mco,\mco)$ is torsion-free, as in the case $c=0$ the ideal $\eta_\lambda(M)$ is just the image of the pairing above.

\subsection*{Structure of $\ecoh^*_A(\mco)$}
A key input in the development of the commutative algebraic properties of the congruence module is a structure theorem for $\ecoh^*_A(\mco)$. The Yoneda product gives $\ext^*_A(\mco,\mco)$ the structure of a graded $\mco$-algebra, and this is inherited by its torsion-free quotient, $\ecoh^*_\lambda(\mco)$. The remarkable fact~\cite[Theorem~6.8]{Iyengar/Khare/Manning:2022a} is that although the Ext-algebra itself can be highly non-commutative, and infinite, $\ecoh^*_\lambda(\mco)$ is just an exterior algebra generated by its degree one component
\[
\ecoh^1_\lambda(\mco)\cong \Hom_\mco(\fp_\lambda/\fp^2_\lambda,\mco)\,.
\]
See \ref{eq:ecoh1} for the isomorphism above.  As explained in \cite[Introduction]{Iyengar/Khare/Manning:2022a}, this may be seen an an integral version of a result, due to Serre, on the structure of the Ext algebra of a regular local ring.  The proof of this structure theorem for $\ecoh^*_\lambda(\mco)$ uses ideas from the theory of differential graded algebras. For the present purpose the important takeaway is that there is an natural isomorphism of $\mco$-modules
\begin{equation}
\label{eq:ecoh-top}
\bigwedge^c \Hom_\mco(\fp_\lambda/\fp_\lambda^2, \mco) \xrightarrow{\ \cong\ } \ecoh^c_\lambda(\mco)\,.    
\end{equation}
The naturality assertion is that given any morphism $\varphi\colon (A,\lambda) \to (A',\lambda')$ in $\acat(c)$, the induced map $\fp_{\lambda}/\fp_{\lambda}^2\to \fp_{\lambda'}/\fp_{\lambda'}$ gives rise to commutative square
\[
\begin{tikzcd}
    \bigwedge^c  \Hom_\mco(\fp_{\lambda}/\fp_{\lambda}^2, \mco) \arrow[r,"\cong"] 
        & \ecoh^c_\lambda(\mco) \\
    \bigwedge^c  \Hom_\mco(\fp_{\lambda'}/\fp_{\lambda'}^2, \mco) \arrow[u] \arrow[r,"\cong"] 
        & \ecoh^c_{\lambda'}(\mco) \arrow[u]
        \end{tikzcd}
\]
of maps of $\mco$-modules. This leads to the following \emph{invariance of domain} property for congruence modules; see \cite[Theorem~7.4]{Iyengar/Khare/Manning:2022a}.

\begin{theorem}
    \label{th:invariance-of-domain}
    Given a surjective map $\varphi\colon (A,\lambda) \to (A',\lambda')$ in $\acat(c)$, and an $A'$-module $M'$ with $\depth_{A'}M'\ge c$, there is a natural isomorphism of $\mco$-modules $\cmod{\lambda'}(M')\cong \cmod{\lambda}(M')$. \qed
\end{theorem}

Since $\vf$ is surjective, $\fp_\lambda \cdot A' = \Ker \lambda'$. The hypotheses in the statement above imply that $\vf_{\fp_\lambda}\colon A_{\fp_\lambda}\to A'_{\fp_\lambda}$ is surjective map of regular local rings of dimension $c$ and hence an isomorphism.

\subsection*{Freeness criterion}
Fix $(A,\lambda)$ in $\acat$ and a finitely generated $A$-module $M$. For any $A$-module $X$ one has a map
\[
\ext^c_A(\mco,X)\otimes_{\mco} (M/\fp_\lambda M) \cong \ext^c_A(\mco,X)\otimes_A M 
	\longrightarrow\ext^c_A(\mco,X\otimes_AM)
\]
where the one on the right is a K\"unneth map. This is functorial in $X$, and one gets the commutative diagram below:
\[
\begin{tikzcd}
\tfree{\ext^c_A(\mco, A)}\otimes_{\mco} \tfree{(M/\fp_\lambda M)} \ar{d} \ar{r} 
	& \tfree{\ext^c_A(\mco, \mco )} \otimes_{\mco} \tfree{(M/\fp_\lambda M)} \ar{d} \\
\tfree{\ext^c_A(\mco, M)} \ar{r} & \tfree{\ext^c_A(\mco, M/\fp_\lambda M)}
\end{tikzcd}
\]
The horizontal maps are one-to-one, by Lemma~\ref{le:injective-at-lambda}. Moreover, the one on right is an isomorphism, as can be verified easily. It follows that the map on the left is one-to-one. This justifies the following result.

\begin{lemma}
\label{le:eq-defect}
\pushQED{\qed}
The diagram above induces a natural surjective map of $\mco$-modules
\[
a_{\lambda}(M)\colon \cmod{\lambda}(A)^\mu \twoheadrightarrow \cmod{\lambda}(M)\,, \quad \text{where $\mu\colonequals \rank_{\lambda}(M)$.}
\]
In particular there is an equality
\[
\length_\mco\cmod{\lambda}(M) = \mu \cdot \length_\mco  \cmod{\lambda}(A) - \length_{\mco}\Ker(a_{\lambda}(M))\,.\qedhere
\]
\end{lemma}

When $A$ is Gorenstein and $M$ is maximal Cohen-Macaulay, $\Ker(a_{\lambda}(M))$ can be interpretted as a``stable" cohomology module of the pair $(A,M)$. This identification leads to the criterion below for detecting free summands of $M$; see \cite[Theorem~9.2]{Iyengar/Khare/Manning:2022a}.

\begin{theorem}
\label{th:freeness}
With notation as above, when $A$ Gorenstein and $M$ is maximal Cohen--Macaulay, $\length_\mco\cmod{\lambda}(M) = \mu \cdot \length_\mco  \cmod{\lambda}(A)$  if and only if 
		\[
		M\cong A^\mu \oplus W \quad\text{and $W_{\fp_\lambda}=0$,}
		\]
as $A$-modules.	In this case, when $\mu\ne 0$ the $A$-module $M$ is faithful. \qed
\end{theorem}

\subsection*{Isomorphism criteria}
The preceding results leads to a criterion for detecting isomorphisms between rings, in terms of congruence modules. 

\begin{proposition}
\label{pr:iso-criteria}
Let $\vf\colon A\to B$ be a surjective map of complete local $\mco$-algebra. Assume there exists an augmentation $\lambda\colon B\to\mco$ such that 
$(A,\lambda\vf)$ and $(B,\lambda)$ are in $\acat(c)$ for some $c\ge 0$, and either of the following conditions hold:
\begin{enumerate}[\quad\rm(1)]
\item
The ring $A$ is Gorenstein, $B$ is Cohen-Macaulay, and
\[
\length_{\mco}\cmod{\lambda\vf}(A) = \length_{\mco}\cmod{\lambda}(B)\,;
\]
\item
The ring $B$ is complete intersection and
\[
\length_{\mco}\con{\lambda\vf}(A) = \length_{\mco}\con{\lambda}(B)\,.
\]
\end{enumerate}
Then the map $\vf$ is an isomorphism.
\end{proposition}

\begin{proof}
(1) The hypotheses imply that $\vf$ is an isomorphism at $\lambda\vf$ so $\rank_{\lambda\vf}B=1$. Thus Theorem~\ref{th:freeness} implies that $B$ is a faithful $A$-module, so $\Ker\vf=(0)$.

(2) is an simple argument using the Jacobi-Zariski sequence arising from maps $A\to B\to \mco$ and Nakayama's Lemma; see \cite[Lemma~5.10]{Iyengar/Khare/Manning:2022a} for details.
\end{proof}

The isomorphism \eqref{eq:ecoh-top} is also a critical input in tracking the behavior of congruence modules under deformations.

\subsection*{Deformations}\label{sec:def}
Fix $(A,\lambda)$  in $\acat(c)$ and elements $\bs{f}\colonequals f_1,\dots,f_n$ in $\fp_{\lambda}$ such that 
their residue classes in the $\mco$-module $\fp_\lambda/\fp_\lambda^2$ form a linearly independent set. Set $\overline{A}\colonequals A/\bs{f}A$. The augmentation $\lambda\colon A\to \mco$ factors through $\overline{A}$ so we an augmentation $\overline{\lambda} \colon \overline{A}\to \mco$. The hypotheses on $\bs{f}$ is equivalent to saying that the pair $(\overline{A},\overline{\lambda})$ is in $\acat(c-n)$; see \cite[\S8]{Iyengar/Khare/Manning:2022a}. A straightforward computation yields an equality
\begin{equation}
\label{eq:deformation-con}
    \length_{\mco}\con{\overline{\lambda}}(\overline{A}) = \length_{\mco}\con{\lambda}(A) + \sum_i\ord(f_i)\,,
\end{equation}
where $\ord(f_i)$ is the order of $f_i$ in $\fp_\lambda/\fp_\lambda^2$, defined by
\[
(\varpi^{\ord(f_i)})\mco = \{\alpha(f_i)| \alpha\in \Hom_{\mco}(\fp_\lambda/\fp_\lambda^2,\mco)\};
\]
 see \cite[\S8.5]{Iyengar/Khare/Manning:2022a}.

\begin{theorem}
\label{th:deformation}
In the context above, let $M$ be a finitely generated $A$-module with $\depth_AM\ge c+1$ and set $\overline{M}\colonequals M/\bs{f}M$. If $\bs{f}$ is $M$-regular, then
    \[
\length_{\mco}\cmod{\overline{\lambda}}(\overline{M}) = \length_{\mco}\cmod{\lambda}(M) + (\rank_\lambda M)\sum_i\ord(f_i)\,.
\]
\end{theorem}

\begin{proof}[Sketch of proof]
It is enough to consider the case when $n=1$. One first reduces to the case when $f$ is not a zerodivisor on $A$ as well; this uses the invariance of domain property for congruence modules, Theorem~\ref{th:invariance-of-domain}. See \cite[Proof of Theorem~8.2]{Iyengar/Khare/Manning:2022a} for details. The essence of the argument is captured $c=1$, so we start with a sketch of the proof in that context.  Since $f$ is in $\fp_\lambda$, and it is not a zerodivisor on $A$ nor on $M$ one gets the isomorphism on the right:
\[
\coker(M\to \Hom(\fp_\lambda,M)) \xrightarrow{ \cong\ } \ext^1_A(\mco,M) \xrightarrow{\ \cong\ } \Hom_{A_0}(\mco, M_0)
\]
The one on the left is from Lemma \ref{le:c=1}. It is straightforward to check that the composite isomorphism is induced by the assignment
\[
\alpha \mapsto -\alpha(f) \mod fM \quad\text{for $\alpha\in \Hom_A(\fp_\lambda,M)$.}
\]
Consider the commutative diagram
\[  
\begin{tikzcd}
\Hom_A(\fp_\lambda,M)\arrow[d, shift right = 2.5ex, "\alpha \mapsto \alpha(f)" swap, "\cong"] \otimes_A
        \Hom_A(M,\mco)\arrow[d, shift left = 8.0ex, "\cong"] \arrow[r] & \Hom_{\mco}(\fp_\lambda/\fp_\lambda^2,\mco)\arrow[d,"\beta \mapsto \beta(f)"] \\
       \phantom{needspace} \overline{M}[\fp_\lambda] \otimes_A  \Hom_{\overline{A}}(\overline{M},\mco)  \arrow[r] &\mco
\end{tikzcd}
\]
The image of the vertical map on the right is precisely the order ideal of $f$, that is to say, $(\varpi^{\ord(f)})$. Since $\cmod{\lambda}(M)$ and $\cmod{\overline\lambda}(\overline M)$ are cokernel of the maps adjoint to the top and bottom row, respectively, the desired equality follows.

To tackle the general case where $c\ge 2$,  consider a commutative diagram analogous to the one above:
\[  
\begin{tikzcd}
\ext^c_A(\mco,M) \arrow[d, shift left = 8.0ex, "\cong"] \otimes_A
        \Hom_A(M,\mco)\arrow[d, shift right = 8.0ex, "\cong"] \arrow[r] 
                        & \ext^c_A(\mco,\mco) \arrow[d] \arrow[r, twoheadrightarrow]
                        	&\ecoh^{c}_{\lambda}(\mco)  \arrow[d]  \\
        \ext^{c-1}_{\overline A}(\mco,\overline M)\otimes_A  \Hom_{\overline A}(\overline M,\mco)  \arrow[r] 
        	& \ext^{c-1}_{\overline A}(\mco,\mco) \arrow[r, twoheadrightarrow]
                        	&\ecoh^{c-1}_{\overline\lambda}(\mco) 
\end{tikzcd}
\]
The isomorphism on the left is by~\cite[Lemma~1.2.4]{Bruns/Herzog:1998}. What is left is to identify the vertical map on the right, and this exploits the isomorphism~\eqref{eq:ecoh-top}.
\end{proof}

\subsection*{Wiles defect}
Fix a pair $(A,\lambda)$  in $\acat$ and a finitely generated $A$-module $M$.  Since $A$ is regular at $\fp_\lambda$, and in particular a domain, the $A_{\fp_\lambda}$-module $M_{\fp_\lambda}$ has a rank. The \emph{Wiles defect} of $M$ at $\lambda$ is the integer
	\[
	\delta_\lambda(M)\colonequals \rank_{\lambda}(M) \cdot \length_{\mco}\con{\lambda}(A) -  \length_{\mco}\cmod{\lambda}(M)\,.
	\]
In particular the Wiles defect of $A$ at $\lambda$ is 
	\[
	\length_{\mco}\con{\lambda}(A) - \length_{\mco}\cmod{\lambda}(A)\,.
	\]
	We refer to \cite[Introduction]{Iyengar/Khare/Manning:2022a} for a discussion on precedents to this definition. 

Theorem~\ref{th:deformation} and \eqref{eq:deformation-con} give the following result, which is \cite[Theorem~8.2]{Iyengar/Khare/Manning:2022a}:

\begin{theorem}
	\label{th:deformation-defect}
	\pushQED{\qed}
	One has $\delta_{\lambda}(M/\bs{f}M)=\delta_{\lambda}(M)$ for $M, \bs f$ as in Theorem~\ref{th:deformation}. \qedhere
\end{theorem}

Also, Theorem~\ref{th:invariance-of-domain} implies the following (which is \cite[Theorem~7.4]{Iyengar/Khare/Manning:2022a}):

\begin{lemma}
	\label{le:defect-invariance}
	If $\vf\colon (A',\lambda')\to (A,\lambda)$ is a surjective map in $\acat(c)$, then
	\[
	\delta_{\lambda'}(M) \ge \delta_{\lambda}(M)
	\]
	with equality if and only if $\con{\lambda'}(A')\cong\con{\lambda}(A)$ holds. \qed
\end{lemma}

With $a_{\lambda}(M)$ as in  \eqref{eq:defect-formula}, one gets a ``defect formula":
\begin{equation}
\label{eq:defect-formula}
\delta_{\lambda}(M) = \rank_{\lambda}(M) \cdot \delta_{\lambda}(A) + \length_{\mco}\Ker(a_{\lambda}(M))\,.
\end{equation}
In particular $\delta_{\lambda}(M)\ge 0$ for all $M$ if and only if $\delta_{\lambda}(A)\ge 0$.

 \begin{theorem}
 \label{th:ci-property}
When $(A,\lambda)\in \acat(c)$ with $\depth A\ge c+1$ one has $\delta_{\lambda}(A)\ge 0$, and equality holds if and only if $A$ is complete intersection.
\end{theorem}

In \cite{Iyengar/Khare/Manning:2022a} this result was proved by reduction to the case $c=0$, using  Theorem~\ref{th:deformation}.  Here is an alternative argument, under the slightly more restrictive case where $A$ is Cohen-Macaulay (so $\dim A = c+1$), that argues by ``going up" to a regular ring. 

\begin{proof}
First we verify that $\delta_{\lambda}(A)=0$ when $A\in \acat(c)$ is complete intersection, that is to say, isomorphic to
\[
\mco\pos{t_1,\dots,t_n}/(f_1,\dots,f_m)
\]
for some regular sequence $\bs{f}\colonequals f_1,\dots,f_m$ in $(\bs{t})$. Since $A$ is in $\acat(c)$ it follows that $n-m=c$ and that $\bs{f}$ satisfies the hypothesis of Theorem~\ref{th:deformation}, so we get the first equality below
\[
\delta_{\lambda}(A) = \delta_{\lambda\eps}(\mco\pos{\bs{t}}) = 0\,.
\]
The second equality is by Theorem~\ref{th:regular}. This is as desired. 

Next we verify that when $A$ is Cohen-Macaulay $\delta_{\lambda}(A)\ge 0$, and that if equality holds $A$ is complete intersection. Since  $\dim A=c+1$ one can find 
a surjection $\eps\colon C\to A$ where $C$ is a complete intersection in $\acat(c)$ and $\eps$ induces an isomorphism $\con{\lambda\eps}(C) \cong \con{\lambda}(A)$; see \cite[Theorem~5.6]{Iyengar/Khare/Manning:2022a}. Thus
\begin{align*}
\length_{\mco}\cmod{\lambda\eps}(C) 
	&= \length_{\mco}\con{\lambda\eps}(C) \\
	&=\length_{\mco}\con{\lambda}(A) \\
	&=\length_{\mco}\cmod{\lambda}(A)\,,
\end{align*}
where the first equality holds because $C$ is complete intersection; the second is by the invariance of domain property~\ref{th:invariance-of-domain}, and the last one is the hypothesis $\delta_{\lambda}(A)=0$. Thus Proposition~\ref{pr:iso-criteria} yields that $\eps$ is an isomorphism.
\end{proof}

Theorem~\ref{th:ci-property} extends to modules, in the following sense; this is \cite[Theorem~9.6]{Iyengar/Khare/Manning:2022a}.

\begin{theorem}
\label{th:defect-modules}
When $\depth_AM\ge c+1$ and $M_{\fp_\lambda}\ne 0$ one has $\delta_{\lambda}(M) \ge 0$, and equality holds if and only if $A$ is complete intersection and 
   \[
		M\cong A^\mu \oplus W \quad\text{and $W_{\fp_\lambda}=0$.}
		\]
\end{theorem}

\begin{proof}[Sketch of proof]
When $\depth A\ge c+1$  also holds, the inequality $\delta_{\lambda}(M)\ge 0$ follows from \eqref{eq:defect-formula} and Theorem~\ref{th:ci-property}. Given this, the other part of the statement follows from Theorems~\ref{th:ci-property} and \ref{th:freeness}.

The argument in the general case is a reduction to the case where $A$ has positive depth, and an induction on $c$. This uses the invariance of domain property and the behavior of defects under deformations, stated below.
\end{proof}

\subsection{$\Lambda$-structures}\label{Lambda-structure}

Motivated by number theory, we consider a setting where the algebra $A$ in $\acat$ has additional structure,  and give  a variant of the computation of change of congruence modules  in  \S \ref{sec:def}  on going modulo regular sequences.

Let $\bs{t}\colonequals t_1,\dots,t_c$ be indeterminates, $\Lambda_c\colonequals \mco\pos{\bs{t}}$ the power-series ring, and $\Lambda_c\to \mco$ the natural augmentation. 
Fix $(A,\lambda)$ in $\acat(c)$ equipped with a finite flat map $\iota\colon\Lambda_c\to A$ of $\mco$-algebras over $\mco$, so that composite map 
\[
\Lambda_c\xrightarrow{\ \iota\ } A\xrightarrow{\ \lambda\ }\mco\,,
\]
is the augmentation (that is, so that $\iota$ is a morphism in $\acat(c)$). Since $\iota$ is flat the sequence $\iota(t_1),\dots,\iota(t_c)$ is $A$-regular. We assume also that the residue classes of $\bs{t}$ in $\fp_\lambda/\fp_\lambda^2$ form a linearly independent set. 
Thus, setting $A_0\colonequals A/\bs{t}A$, the map $\lambda$ factors through $A_0$, yielding an augmentation $\lambda_0\colon A_0\to \mco$, and $(A_0,\lambda_0)$ is in $\acat(0)$. One gets a commutative diagram of $\mco$-algebras
\begin{equation}
\label{eq:algebra-square}
\begin{tikzcd}
\Lambda_c \ar[d,"\eps" swap] \ar[r,"\iota"] & A\ar[d,"\alpha"] \\
\mco \ar[r] &A_0 \ar[r,"\lambda_0"] &\mco
\end{tikzcd}
\end{equation}
all augmented to $\mco$, via $\lambda_0$. We wish to track the change in contangent modules and congruence modules along $\alpha$, and we do that by using the diagram above, to reducing the problem to one about the map $\eps$, where it is trivial, and the map $\iota$, where it is easier to handle.

In the rest of this discussion we write $\fp$ and $\fp_0$ instead of $\fp_{\lambda}$ and $\fp_{\lambda_0}$, respectively.

We first discuss the change in cotangent modules in passing from $\lambda$ to $\lambda_0$. Since $A_0$ is regular at $\lambda_0$, the $\mco$-module $\mathrm{D}_2(\mco/A_0,\mco)$, the second Andr\'e-Quillen homology  of the map $A_0\to \mco$, is torsion. Moreover one has
\[
\mathrm{D}_1(A_0/A,\mco)\cong \mathrm{D}_1(\mco/\Lambda_c,\mco)\cong \fm/\fm^2 \quad\text{where $\fm\colonequals (\bs{t})\Lambda_c$.}
\]
In particular, this is a free $\mco$-module, of rank $c$.  Thus the Jacobi-Zariski sequence arising from the maps $A\to A_0\to \mco$ yields an exact sequence of $\mco$-modules
\[
0\longrightarrow \fm/\fm^2  \xrightarrow{\ \iota\ } \fp/\fp^2 \longrightarrow \fp_0/\fp_0^2 \longrightarrow 0\,.
\]
where we use $\iota$ also to denote the map induced on cotangent modules by the ring map $\iota$. One gets  an exact sequence
\[
0\longrightarrow \fm/\fm^2 \xrightarrow{\ \iota\ } \tfree{(\fp/\fp^2)} 
	\longrightarrow \con{\lambda_0}(A_0)/\con{\lambda}(A) \longrightarrow 0\,.
\]
From a number theory perspective, it is more natural to consider the exact sequence obtained by applying $(-)^*\colonequals \Hom_\mco(-,\mco)$, namely the sequence
\begin{equation}
\label{eq:Kahler}
0\longrightarrow (\fp/\fp^2)^* \xrightarrow{\ \iota^*\ } (\fm/\fm^2)^* \longrightarrow
 \ext^1_{\mco}(\con{\lambda_0}(A_0)/\con{\lambda}(A),\mco)  \longrightarrow 0\,.
\end{equation}
On the subcategory of torsion $\mco$-modules one has an isomorphism of functors
\[
\ext^1_\mco(-,\mco) \cong \Hom_{\mco}(-,E/\mco)\,.\]
where $E$ is the field of fractions of $\mco$. Since the functor on the right preserves lengths, the computations above yield 
\begin{equation}
\label{eq:con-change}
\begin{aligned}
 \length_{\mco}\con{\lambda_0}(A_0) -  \length_{\mco}\con{\lambda}(A) 
 		 & = \length_{\mco}(\con{\lambda_0}(A_0)/\con{\lambda}(A)) \\
		 & = \length_{\mco}\coker(\iota^*)  \\
		 & = \length_{\mco}\coker (\wedge^c\iota^*)\,.
 \end{aligned}
\end{equation}
The second equality holds because $\iota^*$ is a map between free $\mco$-modules of rank $c$.

Now we move on to the congruence modules. Given commutative diagram~\ref{eq:algebra-square} of algebras over $\mco$ and the functorial properties of $\ecoh^{-}_{-}(\mco)$ one gets a commutative diagram of $\mco$-modules
\[
\begin{tikzcd}
\ecoh^c_{\lambda\iota}(\mco) \ar[d,"\cong"] \ar[r,leftarrow] & \ecoh^c_{\lambda}(\mco) \ar[d] \\
\mco=\ecoh^0_{\idmap}(\mco) \ar[r,equal] &\ecoh^0_{\lambda_0}(\mco)=\mco
\end{tikzcd}
\]
The isomorphism in the lower row is clear from the definitions; the vertical isomorphism is by a direct computation. The identity map on $\mco$ is a canonical generator 
for $\ecoh^0_B(\mco)=\Hom_B(\mco,\mco)$, for any $B$ in $\acat$; this is why we write equalities in the last row. Using  the commutative diagram above and the functoriality of the map \eqref{eq:ecoh-top}, one gets a commutative diagram 
\[  
\begin{tikzcd}
	\bigwedge^c (\fp/\fp^2)^* \arrow[r,"\cong"]\arrow[d, "{\wedge^c\iota^*}" swap]
                        & \ecoh^c_{\lambda}(\mco) \arrow[d]   \\
\bigwedge^c (\fm/\fm^2)^* \ar[r,"\cong"] 
		    			& \ecoh^0_{\lambda_0}(\mco) 
\end{tikzcd}
\]

Consider a finitely generated $A$-module $M$ such that $\bs{t}$ is also regular on $M$ and $\depth_AM\ge c+1$. Setting  $M_0\colonequals M/(\bs{t})M$ and using the identifications above, one gets a commutative diagram like so:
\begin{equation}
\label{eq:big-D}
\begin{tikzcd}
\ext^c_A(\mco,M) \arrow[d, shift left = 8.0ex, "\cong"] \otimes_{\mco}
        \Hom_A(M,\mco)\arrow[d, shift right = 8.0ex, "\cong"] \arrow[r]
                        	& \bigwedge^c (\fp/\fp^2)^* \arrow[d,"{\wedge^c\iota^*}"] \\
        \ext^{0}_{A_0}(\mco,M_0)\otimes_{\mco}  \Hom_{A_0}(M_0,\mco)  \arrow[r] 
							&\bigwedge^c (\fm/\fm^2)^* 
\end{tikzcd}
\end{equation}

All these lead to the following structural refinement of Theorem~\ref{th:deformation}.

\begin{theorem}
\label{th:deformation-with-lambda}
Viewing $\eta_{\lambda}(M)$ and $\eta_{\lambda_0}(M_0)$ as submodules of $\wedge^c(\fp/\fp^2)^*$ and $\wedge^c(\fm/\fm^2)^*$, respectively, there is an equality 
\[
\eta_{\lambda_0}(M_0)=(\wedge^c\iota^*)(\eta_\lambda(M))\,.
\] 
Moreover, with $\mu\colonequals \rank_\lambda M$ there are equalities
\begin{align*}
\length_{\mco}\cmod{\lambda_0}(M_0)- \length_{\mco}\cmod{\lambda}(M) 
			&= \mu\cdot \length_{\mco}\coker(\wedge^c\iota^*)\\
			&= \mu\cdot (\length_{\mco}\con{\lambda_0}(A_0) - \length_{\mco}\con{\lambda}(A))\,.
\end{align*}
\end{theorem}

\begin{proof}
The first part of the proposition is immediate from the commutative diagram~\eqref{eq:big-D}. The second part then follows, given also \eqref{eq:con-change}.
\end{proof}

\section{Zeta lines  and congruence modules}
\label{se:cotorsion}

We focus on number theory applications  of the results in Section~\ref{se:ca-section}, notably the exact sequence \eqref{eq:Kahler} and Theorem \ref{th:deformation-with-lambda}.  The main result is Theorem~\ref{th:zeta}. We begin with  Proposition \ref{prop:SelmerChange}, which is  a simple consequence of the Poitou-Tate exact sequence and is used to prove Proposition \ref{pr:duality}.

Let $F$ be a number field,  $S$ a finite set of places of $F$, and $G_{F,S}$ the  Galois group of  $F_S/F$, the maximal extension of $F$ unramified outside the places above $S$ in an algebraic closure of $F$.    Fix a prime number $p$, a finite extension $E/\bbQ_p$, and  let $\mco$ denote the ring of integers of $E$. Let $A$ be a $\mco$-module, which is finitely or cofinitely generated, with an action of $G_{F,S}$. The Pontryagin dual and the twisted Pontryagin dual of $A$, respectively, are the $G_{F,S}$-modules
\[
A^\vee\colonequals \Hom_{\mco}(A,E/\mco)\quad\text{and}\quad  A'=A^\vee(1)=\Hom_{\mco}(A,E/\mco(1))\,.
\]

A \emph{Selmer datum} for $S$ and $A$ is a collection $\cL=\{\cL_v\}_v$, where $\cL_v$ is an $\mco$-submodule of $H^1(G_v,A)$ for each $v\in S$. The corresponding \emph{Selmer group} is
\[ 
H^1_\cL(F,A)\colonequals \Ker\big(H^1(G_{F,S},A)\longrightarrow \prod_{v\in S}H^1(G_v,A)/\cL_v\big)\,.
\]
Local Tate-duality induces the perfect pairing 
\begin{equation}
\label{eq:TateLocalDuality}
H^1(G_v,A)\times H^1(G_v,A')\to H^2(G_v,E/\mco(1))\cong E/\mco.
\end{equation}
The \emph{dual Selmer datum} $\cL^\perp$ (for $S$ and $A'$) is defined with $\cL^\perp_v\subset H^1(G_v,A')$ the annihilator of $\cL_v$ under this pairing. The \emph{dual} Selmer group of $A$ is $H^1_{\cL^\perp}(F,A')$.

For $i=1,2$, set $\Sha_S^i(F,A)\colonequals \Ker\big(  H^i(G_{F,S},A)\to \prod_{v\in S}H^i(G_v,A) \big)$.
The result below is standard; see  \cite[8.7.9]{Neukirch/Schmidt/Wingberg:2008}. The argument is based on notes of Boeckle.

\begin{lemma}
\label{lem:Greenberg-Wiles}
\pushQED{\qed}
One has an exact sequence
\[
0 \to H^1_\cL(F,A)\to H^1(G_{F,S},A)\to \prod_{v\in S}\frac{H^1(G_v,A)}{\cL_v} \to  H^1_{\cL^\perp}(F,A’)^\vee\to\Sha_S^2(F,A)\to 0\,.
\]
If $A$ is finite, then 
\[
 \frac{\#H^1_\cL(F,A)}{\#H^1_{\cL^\perp}(F,A')}=\frac{\#H^0(F,A)}{\#H^0(F,A')}\cdot \prod_{v\in S}\frac{\#\cL_v}{\# H^0(F_v,A)}\,. \qedhere
\]
\end{lemma}

Given Selmer datum $\cL$ and $\tcL$ for $S$ and $A$,  we write $\cL\subseteq\tcL$ if $\cL_v\subseteq \tcL_v$ for all~$v$.

\begin{proposition}
\label{prop:SelmerChange}
Suppose that $A$ is compact or cocompact and that $\cL\subseteq \tcL$ are Selmer data. Then one has a natural exact sequence
\[ 
0\longrightarrow \frac{H^1_{\tcL}(S,A)}{H^1_\cL(S,A)} \longrightarrow \prod_{v\in S} \frac{\tcL_v}{\cL_v}
	\longrightarrow \bigg( \frac{H^1_{\cL^\perp}(S,A')}{H^1_{\tcL^\perp}(S,A')}\bigg)^\vee \longrightarrow 0
	\]
\end{proposition}

\begin{proof}
From the definitions one gets that the following natural maps are injective:
\[
H^1_{\cL}(S,A) \longrightarrow H^1_{\tcL}(S,A)\quad\text{and}\quad H^1_{\tcL^\perp}(S,A')\longrightarrow H^1_{\cL^\perp}(S,A')\,.
\]
Setting $K\colonequals \Ker (H^1_{\cL^\perp}(F,A')^\vee\to\Sha_S^2(F,A)) $ and similarly $\widetilde{K}$ with $\tcL$ in place of $\cL$, the naturality of the exact sequence in Lemma~\ref{lem:Greenberg-Wiles} yields the commutative diagram
\[
\begin{tikzcd}
0 \arrow[r] & \displaystyle{\frac{H^1(G_{F,S},A)}{H^1_{\cL}(F,A)}} \arrow[d,twoheadrightarrow] \arrow[r] 
	& \prod_{v\in S} \displaystyle\frac{H^1(G_v,A)}{\cL_v} \arrow[d,twoheadrightarrow] \arrow[r] 
		& K \arrow[d,twoheadrightarrow] \arrow[r] & 0 \\
0 \arrow[r] & \displaystyle{\frac{H^1(G_{F,S},A)}{H^1_{\tcL}(F,A)}} \arrow[r] 
    & \prod_{v\in S}\displaystyle \frac{H^1(G_v,A)}{\tcL_v} 
			\arrow[r] & \widetilde{K} \arrow[r] & 0 	
\end{tikzcd}
\]
The Snake Lemma yields  the exact sequence
\[
0\longrightarrow  \displaystyle \frac{H^1_{\tcL}(F,A)}{ H^1_\cL(F,A)} \longrightarrow 
		\prod_{v\in S} \displaystyle \frac{\tcL_v}{\cL_v} \longrightarrow K \longrightarrow \widetilde{K} \longrightarrow 0\,.
\]
By applying the Snake Lemma to the  commutative diagram
\[
\begin{tikzcd}
0  \arrow[r] & K \arrow[d,twoheadrightarrow] \arrow[r]
				& H^1_{\tcL^\perp}(F,A')^\vee \arrow[d,twoheadrightarrow] \arrow[r] &\Sha_S^2(F,A) \arrow[d,equal] \arrow[r] & 0 \\
0  \arrow[r] & \widetilde{K} \arrow[r]
				& H^1_{\cL^\perp}(F,A')^\vee \arrow[r] &\Sha_S^2(F,A) \arrow[r] & 0 				
\end{tikzcd}
\] 
and using Pontryagin duality, we obtain isomorphisms
\begin{align*}
\Ker (K\to \widetilde{K})
		&\cong \Ker(H^1_{\cL^\perp}(F,A')^\vee\to H^1_{\tcL^\perp}(F,A')^\vee ) \\
		&\cong \bigg(\frac{H^1_{\cL^\perp}(F,A')}{H^1_{\tcL^\perp}(F,A')}\bigg)^\vee,
\end{align*}
concluding the proof of the proposition.
\end{proof}

\subsection*{Nearly ordinary Hilbert modular forms}
We specialize to a situation which corresponds to that of \cite{Urban:2021}.  Thus $F$ is a totally real number field, $p$ an odd prime and $f$ a nearly ordinary at $p$, holomorphic 
 and cohomological  cuspidal Hilbert modular newform for $\GL_2(\bbA_F)$. Our coefficients $A$ as  in the section above arise from the adjoint  $\ad \rho_f$  (and $\ad (\rho_f \otimes_\mco E/\mco)$) of  (an) integral Galois  representation 
 \[
 \rho_f\colon G_{F,S} \to \GL_2(\mco)
 \] 
 associated  to $f$ and an embedding of $\overline{\bbQ} \hookrightarrow \overline{\bbQ_p}$.  We assume the residual representation  $
 \overline{\rho}_f\colon G_{F,S} \to \GL_2(k)$ is irreducible (which implies that it is absolutely irreducible as $\rho_f$ is totally odd and $p>2$), and thus there is a unique integral representation $\rho_f$ associated to $f$ (by a well-known result of Carayol).

We apply the results of the previous section  to study the Galois cohomology of the adjoint representation of $\rho_f$ with several local conditions at places dividing $p$. For each $v|p$, we fix  a decomposition subgroup $D_v$ at $v$. We call $I_v\subset D_v$ the inertia subgroup, $F_v$ the completion of $F$ at $v$, and $d_v$  its degree over $\bbQ_p$. By nearly ordinarity of $f$, for each $v|p$ there exists $g_v\in \GL_2(\mco)$ such that the restriction to the decomposition subgroup $D_v$ at $v$ of $g_v\rho_fg_v^{-1}$ is upper triangular. We also assume  that it is $v$-distinguished (that is, the characters appearing on the diagonal are distinct  modulo the uniformizer $\varpi$ of $\mco$) and indecomposable.  We then consider the following summands:
\[
\CF^+_v\colonequals  \{ g_v  \begin{pmatrix} 0&* \\ 0&0\end{pmatrix} g_v^{-1} \} \subset \CF^0_v\colonequals  \{ g_v  \begin{pmatrix} *&*\\0&*\end{pmatrix}g_v^{-1} \} \subset \ad\rho_f
\]
We denote $Gr^0_v\colonequals\CF^0_v/\CF^+_v$ and fix an isomorphism of $D_v$-modules $Gr^0_v\cong \mco$.  Let $B$ be a $\mco$-module.
The \emph{ordinary} condition $H^1_{\ord}(F_v,\ad\rho_f\otimes B)$  at $v$ is given by  the image of $H^1(F_v, \CF_v^0\otimes B)$ in $H^1(F_v,\ad\rho_f\otimes B)  $, and (in the terminology of Wiles~\cite{Wiles:1995}) the \emph{Selmer} condition $H^1_{\Sel}(F_v,\ad\rho_f\otimes B)$ at $v$ is given by the image of 
\[
\Ker (H^1(F_v,\CF^+_v\otimes B)\longrightarrow H^1(I_v, Gr^0_v\otimes B))
\]
in $H^1_{\ord}(F_v,\ad\rho_f\otimes B)$. Since the representation $\rho_f$ is $v$-distinguished, we get an exact sequence:
\[
0\to H^1_{\Sel}(F_v,\ad\rho_f\otimes B)\to H^1_{\ord}(F_v,\ad\rho_f\otimes B)\to H^1(I_v,B)^{\frac{D_v}{I_v}}\to 0
\]
The orthogonal of the finite \emph{Selmer} condition    $H^1_{\Sel^\perp}(F_v,\ad\rho_f\otimes B(1))$ and of the ordinary condition $H^1_{\ord^\perp}(F_v,\ad\rho_f\otimes B(1))$ are respectively given by the images  of  $H^1(F_v,\mathcal F_v^0\otimes B(1))$ and  of $H^1(F_v,\mathcal F^+_v\otimes B(1))  $ in $H^1(F_v,\ad\rho_f\otimes  B(1))$.

\begin{remark} 
\label{BK=Se} If the action of $D_v$ on $\mathcal F^+_v$ is distinct from the cyclotomic character, then the finite \emph{Selmer} condition is nothing else but the finite  Bloch-Kato condition.
In that case  one has
\begin{align*}
H^1_{\Sel}(F_v,\ad\rho_f\otimes B) &=H^1_{f}(F_v,\ad\rho_f\otimes B) \\
H^1_{\Sel^\perp}(F_v,\ad\rho_f\otimes B(1)) &=H^1_{f}(F_v,\ad\rho_f\otimes B(1))\,.
	\end{align*}
\end{remark}

Next we interpret some higher congruence modules in terms of Galois cohomology
 (local and global) and by applying Theorem \ref{th:deformation-with-lambda} in the situation described below. 
 
\subsection*{Congruence modules and Galois cohomology}
We recall some ingredients of the set up of \cite{Urban:2021}; most of the notation is borrowed from it.

Let  $\kappa\colonequals (\sum_\sigma k_\sigma .\sigma,\sum_\sigma l_\sigma .\sigma,  ) \in\bbZ[\Sigma_F]^2$ be the  weight of the 
cohomological Hilbert modular cusp form $f$. We have  $k_\sigma\geq 2$ for all $\sigma\in \Sigma_F$ 
and $w=k_\sigma+2l_\sigma$  is independent of $\sigma$. For such a weight and a $O_{F'}$-algebra $S$ with $F'$ the normal closure of $F$, we consider  the algebraic representation of $\GL_2(O_{F’})$:
\[
L(\kappa,S)\colonequals\bigotimes_\sigma \mathrm{Sym}^{k_\sigma-2}(S^2)\otimes {\det}^{l_{\sigma}}\,.
\]

For each neat open compact subgroup $K\subset \GL_2(\bbA_f\otimes F)$, this defines a local system $L(\kappa,\mathbb{C})$ on the Hilbert modular variety
\[
X(K)\colonequals\GL_2(F)\backslash \GL_2(\bbA\otimes F)/KK_\infty Z(F)
\]
where $Z$ stands for the center of $\GL_2$ and $K_\infty$ is the maximal compact (modulo the center) subgroup of $\GL_2(\mathbb{R}\otimes F)$.

Let $\mathfrak n\subset O_F$ be the tame conductor of $f$. It is a nonzero  integral ideal of $O_F$ prime to $p$. Let $K^p_{11}(\mathfrak n) \subset \GL_2(\widehat\bbZ^p\otimes O_F)$ be the subgroup of matrices which are upper unipotent modulo $\mathfrak n$  and where we have written $\widehat\bbZ^p$ for the prime-to-$p$ part of the profinite completion of  $\bbZ$. We will assume that $K_{11}^p(\mathfrak n)$ is neat. Let $\omega$ be the central character of the cuspidal representation 
attached to $f$. It is an id\` ele class character of conductor dividing $\mathfrak n p^\infty$ and infinity type $|\cdot |^w$.

For each positive  integer $n$, we denote $K_0(p^n)$ the subgroup of $\GL_2(O_F\otimes\bbZ_p)$ of matrices 
which are upper triangular modulo $p^n$ and by  $K_1(p^n)$ its subgroup of those such that the diagonal entries are congruent modulo $p^n$. We identify 
$K_0(p^n)/K_1(p^n)$ with $(O_F/p^nO_F)^\times$
via the map $\left(\begin{array}{cc} a&b\\c&d\end{array}\right)\mapsto a^{-1}d$. 

Let $h_{\kappa}^{\ord}(\mathfrak np^n,\omega)$  be the nearly ordinary Hecke algebra 
of level $K_{11}^p(\mathfrak n)K_1(p^n)$, weight $\kappa$ and central character $\omega$. We then consider the universal nearly ordinary Hecke algebra of weight $\kappa$ and tame level $K^p$ 
and action of the center given by $\omega$.
\[
\mathbf h^{\ord}= \mathbf h^{\ord}_{\kappa}(\mathfrak n)\colonequals\lim_{\underset{n}{\leftarrow}} h_{\kappa}^{\ord}(\mathfrak np^n)
\]
The Hecke ring $h_{\kappa}^{\ord}(\mathfrak np^n)$ has a natural structure of $\mco[(O_F/p^nO_F)^\times]$-algebra which induces a structure of $\Lambda_F$-algebra 
on $\mathbf h^{\ord}$ with 
\[
\Lambda_F\colonequals\mco\pos{O_{F,p}^1}\cong \mco\pos{t_1,\dots,t_d}
\]
and $O_{F,p}^1\cong \bbZ_p^d$ the subgroup of $O_{F,p}^\times$ of local units congruent
 to 1 modulo $p$. Let $\fm$ be the kernel of the map $\Lambda_F\to \mco$ corresponding to the trivial character of $O_{F,p}^1$.

The Hecke eigensystem attached to our nearly ordinary Hilbert modular form $f$ gives us an homomorphism:`1
\[
\lambda_f \colon \mathbf h^{\ord}\rightarrow  h^{\ord}_\kappa(\mathfrak n p^r) \rightarrow \mco.
\]
with $r$ the smallest integer so that $f$ is $K_1(p^r)$-invariant.

We now denote by $\bbT^{\ord}$ (resp. $\bbT_0$) the localisation of $\mathbf h^{\ord}$ (resp. $ h^{\ord}_\kappa(\mathfrak n p^r) $ )    
at its maximal ideal $\fm_f$ containing $\ker \lambda_f$. It is known thanks to the work of Hida that $\bbT^{\ord}$ is free of finite rank over $\Lambda_F$. Moreover, we have a canonical isomorphism
\[
\bbT^{\ord}\otimes_{\Lambda_F}\mco\cong \bbT_0\,.
\]
We now construct a $\bbT^{\ord}$-module which is  free over $\Lambda_F$ and interpolates the nearly ordinary cohomology of the Hilbert modular variety localized at the maximal ideal associated to $f$.
For any $p$-adically complete $\mco$-algebra $A$ and $n\geq r$, let
$$\mathcal C_n(\kappa,A)\colonequals Ind_{K_1(p^n)}^{K_1(p^r)} L(\kappa,A)$$
and write $\mathcal C(\kappa,A)$ for the direct limit of  the $\mathcal C_n(\kappa,A)$ for the obvious transition maps, and $\mathcal C(\kappa,\mco)$
 for the inverse limit of the $\mathcal C(\kappa,\mco/p^m\mco)$ as $m$ varies. 
 It is clearly a $\Lambda_F[K_1(p^r)]$-module.
 
Let $\eta_{\lambda_f}(M)$ be the congruence ideal of $M$ with respect to $\lambda_f$ introduced in \ref{sec:pairing}.

\begin{proposition} 
\label{pr:control} 
Assume the image of $\bar\rho_f$ is not solvable. Then for any $\epsilon$ in $\{\pm1\}^{\Sigma_F}$, the $\bbT^{\ord}$-module
	\[
	\mathbf M^\epsilon\colonequals H^d(X(K_{11}(\mathfrak np^r), \mathcal C(\kappa,\mco))_{\fm_f} ^\epsilon
	\]
	is free of finite rank over $\Lambda_F$, and
	\[
	\mathbf M^\epsilon/\mathfrak m \mathbf M^\epsilon= \mathbf M^\epsilon_0\colonequals H^d(X(K_1(\mathfrak np^r), \mathcal L(\kappa,\mco))_{\fm_f}^\epsilon\,.
	\]
	Moreover $\eta_{\lambda_f}(\mathbf M^\epsilon_0)=(\xi^\epsilon_f)$,  where 
		\[
		\xi_f^\epsilon\colonequals\frac{\Gamma(\ad\rho_f,1)L^{S_f}(\ad\rho_f,1)}{\Omega_f^\epsilon\Omega_f^{-\epsilon}}
		\]
	where $S_f$ is the set of finite places where $\rho_f$ is ramified and $(\Omega_f^\epsilon)_{ \epsilon\in \{\pm 1\}^{\Sigma_F}}$ 
	are the canonical complex periods attached to the Hilbert modular form $f$ in\footnote{In \cite{Dimitrov:2009}, the  $\epsilon$-parts of the cohomology and the  periods are indexed by 
	the subsets $J\subset \Sigma_F$ corresponding to the character $\epsilon_J$} \cite[\S 7.1]{Dimitrov:2009}. 
\end{proposition}

\begin{proof}
This is a classical exercise in Hida theory since the localization at $\mathfrak m_f$ capture a direct factor of the {\it nearly ordinary} part of the cohomology. The fact that the module is free over $\Lambda_F$ follows from a control theorem  and  the vanishing Theorem 7.1.1 of  Caraiani and Tamiozzo
\cite{Caraiani/Tamiozzo:2023}. The last part of the proposition follows from a computation of Dimitrov in the sections 7.2 and 7.3 of \cite{Dimitrov:2009},and in particular its equations (50) and (51). 
\end{proof}

We use the inclusion $ \eta_{\lambda_f}(\mathbf M^\epsilon)\subset  F^d_{\lambda_f}(\mco) = \bigwedge_{\mco}^d(\fp/\fp^2)^*$ to define  the zeta $\mco$-module associated to $f$.

\subsection*{Construction of zeta lines} 
Given an $\mco$-module $A$ we set
\[
A^*\colonequals \Hom_{\mco}(A,\mco)\,.
\]
Let $R^{\ord}$ (resp. $R_0$) be the universal deformation ring of $\bar\rho_f$ with fixed determinant equal to $\det\rho_f$ and with nearly ordinary conditions  (resp. with ordinary condition of  weight $\kappa$) at places dividing $p$  and  the unramified condition at finite places away from those dividing $\mathfrak n p \infty$. 

We have a canonical surjective map $R^{\ord} \to \bbT^{\ord}$. Set
\[
\fp \colonequals\Ker (\bbT^{\ord}\to \mco) \quad \text{and} \quad  \fp_R \colonequals\Ker ( R^{\ord}\to \mco)\,,
\]
and consider the natural surjection
\[
H^1_{\full,\ord}(F,\ad\rho_f\otimes E/\mco)^\vee \cong \fp_R/\fp_R^2 \twoheadrightarrow \fp/\fp^2\,.
\]
The subscript $\full$ means that no local conditions are required at places dividing $\mathfrak n$. The isomorphism is standard;  see, for example, \cite[Lemma 3.3]{Urban:2021}. This induces maps
\[
F^d_{\lambda_f}(\mco) = \bigwedge_{\mco}^d(\fp/\fp^2)^* \longrightarrow \bigwedge_{\mco}^d(\fp_R/\fp_R^2)^*\cong \bigwedge^d_{\mco} H^1_{\full,\ord}(F,\ad(\rho_f))\,.
\]
\begin{definition}
The image of the submodule  $\eta_{\lambda_f}(\mathbf M^\epsilon)\subset   \bigwedge^d ( \fp/\fp^2)^*$ under the composition of maps above is a cyclic $\mco$-submodule; we write  it as $(z_f^\epsilon)$ and call it the zeta line. Thus $z_f^\epsilon$   is well-defined only up to multiplication by a unit in $\mco$.
\end{definition}

For each $v$, the quotient map $\CF^0_v\to Gr^0_v$ induces the map
\[
H^1(F_v,\CF^0_v) \longrightarrow H^1(F_v,Gr^0_v)\cong  H^1(I_v,Gr^0_v)^{D_v/I_v}\cong  \mco^{d_v}\,.
\]
Since $\sum_{v|p}d_v=[F:\bbQ]=d$ these induce the isomorphism
\[
\bigwedge^d(\prod_{v|p} H^1(F_v,\CF^0_v)) \xrightarrow{\ \cong \ } \bigotimes_{v|p}\bigwedge_\mco^{d_v} H^1(I_v,Gr^0_v)^{D_v/I_v}\cong \mco\,.
\]
Precomposing this with the $d$th exterior power of the restriction map 
\[
\mathrm{res}_p\colon H^1_{\full,\ord}(F,\ad\rho_f) \to \prod_{v|p} H^1(F_v,\CF^0_v)\,.
\] 
yields the map
\[
\bigwedge ^d \mathrm{res}_p \colon  \bigwedge^d H^1_{\full,\ord}(F,\ad\rho_f) \longrightarrow \bigotimes_{v|p}\bigwedge_\mco^{d_v} H^1(I_v,Gr^0_v)^{D_v/I_v}\,.
\]
The following  theorem is the main result of this section.

\begin{theorem}
\label{th:zeta} 
Assume that the residual representation $\bar\rho_f$ has non solvable image and choose
  $\epsilon\in\{\pm1\}^{\Sigma_F}$.
	Then \[(\bigwedge ^d \mathrm{res}_p) (z^\epsilon_f) =(\xi_f^\epsilon)\] where  as before \[
		\xi_f^\epsilon\colonequals\frac{\Gamma(\ad\rho_f,1)L^{S_f}(\ad\rho_f,1)}{\Omega_f^\epsilon\Omega_f^{-\epsilon}}.
		\]
\end{theorem}

\begin{proof}
Let $R_0$ be the universal deformation ring of $\bar\rho_f$ with fixed determinant equal to $\det\rho_f$ and with ordinary condition of  weight $\kappa$ at places dividing $p$  and  the unramified condition at finite places away from those dividing $\mathfrak n p \infty$, and $R_0 \to \bbT_0$ the canonical surjection. The restriction of the universal deformations to the decomposition subgroups at places dividing $p$ 
gives an homomorphism  $\Lambda_F\to R^{\ord}$ making the map  $R^{\ord} \to \bbT^{\ord}$ an $\Lambda_F$-algebra homomorphim and a canonical 
isomorphism $R^{\ord}\otimes_{\Lambda_F}\mco\cong R^{\ord}_0$. Setting
\[
 \fp_0 \colonequals\Ker (\bbT^{\ord}_0\to \mco)=\fp/\fm  \quad \text{and} \quad \fp_{R,0} \colonequals \Ker  (R_0^{\ord}\to \mco)=\fp_R/\fm
\]
the isomorphism and surjective maps induces the following commutative diagram of K\"ahler differentials (see  \S \ref{Lambda-structure}).
\[
\begin{tikzcd}[column sep=tiny]
    0 \arrow[r] & \fm/\fm^2 \arrow[d,equal] \arrow[r]  & \fp/\fp^2 \arrow[r]  & \fp_{0}/\fp_{0}^2  \arrow[r]  & 0 \\
    0 \arrow[r] & \fm/\fm^2 \arrow[d,equal] \arrow[r]  & \fp_{R}/\fp_{R}^2  \arrow[u]  \arrow[d,equal]\arrow[r]  & \fp_{R,0}/\fp_{R,0}^2  \arrow[u] \arrow[d,equal] \arrow[r]  & 0 \\
      0 \arrow[r] & \bigoplus_{v|p }(H^1(I_v,\frac{E}{\mco})^{\frac{D_v}{I_v}})^\vee\arrow[r]  & H^1_{\full,\ord}(F,\ad\rho_f\otimes \frac{E}{\mco})^\vee\arrow[r]  &  H^1_{\full,\Sel} (F,\ad\rho_f\otimes \frac{E}{\mco})^\vee\arrow[r]  & 0 
\end{tikzcd}
\]
Here  $H^1_{\full,\Sel} (F,\ad\rho_f\otimes \frac{E}{\mco})$ means no condition at primes dividing $\mathfrak n$, and the Selmer condition at places of $F$ above $p$.  The exactness on the left in the top row follows from  Hida’s theorem  that $\bbT^{\ord}$ is unramified over the weight space $\Lambda_F$  at the augmentation $\lambda_f$ arising from the holomorphic cohomological newform $f$. The vertical arrows are surjective and that  the $\mco$-module $\fp_0/\fp_0^2\cong \Omega_{\bbT_0/\mco}\otimes_{\lambda_f}\mco$ is torsion, and therefore  both $\fm/\fm^2$ and  $(\fp/\fp^2)^{t\mathrm f}$ are free of rank $d$ over $\mco$. The diagram above yields the commutative diagram
\[
\begin{tikzcd}
\bigwedge^d ( \fm/\fm^2)^*\arrow[d,equal] &\arrow[l]   \bigwedge^d ( \fp/\fp^2)^*\arrow[d]   \\
 \bigwedge^d(\fm/\fm^2)^* \arrow[d,equal]& \arrow[l] \bigwedge^d( \fp_{R}/\fp_{R}^2 )^*\arrow[d,equal]  \\
 \bigotimes_{v|p }\bigwedge^{d_v}H^1(I_v,{\mco})^{\frac{D_v}{I_v}}  &\arrow[l]   \bigwedge^d H^1_{\full,\ord}(F,\ad\rho_f) 
\end{tikzcd}
\]
Given this diagram, Proposition  \ref{pr:control}, and Theorem  \ref{th:deformation-with-lambda}, it follows that the image of $z_f^\epsilon$ under the local restriction map at $p$ is $\xi_f^\epsilon$,  fixing an isomorphism with $\mco$.
\end{proof}

\begin{remark} Our hypothesis here are less restrictive than in \cite{Urban:2021}. Moreover, the method used here allows us to bypass the use of local complete intersection results  on the corresponding Hecke ring used in \cite{Urban:2021}, and therefore to remove some hypotheses; in particular, it does not require us to have a $R=\bbT$  theorem.  

As explained in  \cite{Urban:2021},  $(z_f^\epsilon)$ is the bottom class of an Euler system of rank $d$. It would be interesting to extend our new method to construct the other classes using higher congruence modules for the base change of $f$ to abelian extensions of $F$.  We have shown that  $(z_f^\epsilon)$  is related to the $L$-value $\xi_f$. If we could extend our method we would be able to construct  the $p$-adic $L$-function  $L^{S_f}_p(\ad\rho_f,s)$.
\end{remark}

\begin{remark} 
By \eqref{eq:Kahler}, cokernel of $(\fp/\fp^2)^* \to (\fm/\fm^2)^*$ has length equal to 
\[
\length_\mco(\Phi_{\lambda_f}(\bbT)) -\length_\mco(\Phi_{\lambda_f}(\bbT^{\ord}))\,.
\]  
This results in ``factorizing’’ classical Selmer groups $\Phi_\lambda(R)$  into a part coming from $\Phi_\lambda(R^{\ord})$ and a part coming from the cokernel of  $(\fp/\fp^2)^* \to (\fm/\fm^2)^*$. 
\end{remark}

\subsection*{The co-torsion in Galois cohomology}
In this section, we assume that $\rho_f$ is a minimal deformation of $\bar\rho_f$ as in \cite[\S4.2]{Dimitrov:2009}. We replace the rings $R^{\ord}$ and $R_0$ by their minimal deformation analogues.  Then it is known that the maps $R_{\mathrm{min}}^{\ord}\to\bbT^{\ord}$ and $R_{0,\mathrm{min}}\to\bbT_0$ are isomorphisms of  complete intersection rings thanks to the works of Fujiwara \cite{Fujiwara} and  Dimitrov \cite{Dimitrov:2009} (note that we can and do replace the big image assumption  of Dimitrov for $\bar\rho_f$ by the much weaker one of \cite{Caraiani/Tamiozzo:2023} of being non solvable). In addition, we make the following hypothesis justified by the Remark \ref{BK=Se}. For all  $v|p$, we assume that the action of $D_v$ on $\mathcal F^+_v$ is distinct from  the cyclotomic character.  By our minimality assumption, the local conditions at places away from $p$ are the finite Bloch-Kato conditions for all the Galois cohomology 
groups considered in this section, so that we can now make the following identifications:

\begin{itemize}
\item  $ \Hom(\fp/\fp^2,\mco)=H^1_{\ord}(F,\ad \rho_f)$
\item $ \Hom(\fm/\fm^2,\mco)=H^1_{\ord/f}(F_p,\ad \rho_f)=\oplus_{v|p}H^1(I_v,\mco)^{\frac{D_v}{I_v}}$
\item $ H^1_{f}(F,\ad \rho_f \otimes_\mco E/\mco)=\fp_0/\fp_0^2\colonequals\Phi_{\lambda_f}$ since $\bbT_0$ is finite over $\mco$,
\item $ {\mathrm{cotors}}(H^1_{\cL^{\ord}}(\bbQ,\ad \rho_f\otimes_\mco E/\mco))=\tors(\fp/\fp^2)=\Phi_{\lambda_f}^{\ord}\colonequals\Phi_{\lambda_f}(\bbT^{\ord})$
\end{itemize}

We abbreviate $\eta_{\lambda_f}(\bbT_0)$ and $\eta_{\lambda_f}(\bbT^{\ord})$ to  $\eta_{\lambda_f}$ and $\eta_{\lambda_f}^{\ord}$, respectively, and view them as ideals of $\mco$.   Since we have assumed that the Hecke rings are complete intersection we have
\[
 \eta_{\lambda_f}= \mathrm{Fitt}_\mco  (\Phi_{\lambda_f})\;
\subset\;
  \eta_{\lambda_f}^{\ord}= \mathrm{Fitt}_\mco  (\Phi_{\lambda_f}^{\ord}) 
\]
Here is an interpretation of the invariants $\Phi^{\ord}_{\lambda_f}$ and $\Psi^{\ord}_{\lambda_f}$.

\begin{proposition}
\label{pr:duality} 
With the minimality  assumptions as  above  we have
\begin{enumerate}[\quad\rm(1)]
\item the isomorphisms:
\[
\Phi_{\lambda_f}^{\ord}\cong {\mathrm{cotors}}(H^1_{{\ord}}(F,\ad \rho_f \otimes_\mco E/\mco))
\] 
and an equality $\length_\mco( \Phi_{\lambda_f}^{\ord})=\length_\mco(H^1_{ {\ord}^\perp}  (F,{\ad \rho_f} \otimes_\mco E/\mco(1)))$.
\item an isomorphism:
\[
\Psi_{\lambda_f}^{\ord}\cong \frac{ \bigwedge^d H^1_{\ord}(F,\ad\rho_f) }{ (z_f ^\epsilon) }\,.
\]
\end{enumerate}
\end{proposition}

\begin{proof}
From \eqref{eq:Kahler} we get the exact sequence
\[
0 \longrightarrow (\fp/\fp^2)^*\longrightarrow (\fm/\fm^2)^*
	\longrightarrow  \left( \frac{\Phi_{\lambda_f}}{ \Phi^{\ord}_{\lambda_f}}\right)^\vee \longrightarrow 0\,.
\]
Comparing with the following Poitou-Tate duality exact sequence
\[ 
0 \to H^1_{{\ord}}  (F, \ad \rho_f) \to  H^1_{\ord /f}(F_p,\ad\rho_f)
\to \left(\frac{ H^1_f(F,{\ad \rho_f} \otimes_\mco E/\mco(1))}   { H^1_{ {\ord}^\perp}  (F,{\ad \rho_f} \otimes_\mco E/\mco(1))}\right)^\vee \to 0
\] 
from Proposition  \ref{prop:SelmerChange} and the identifications recalled above, we get
\[
 \frac{\Phi_{\lambda_f}}{ \Phi^{\ord}_{\lambda_f}}  \cong
\frac{ H^1_f(F,{\ad \rho_f} \otimes_\mco E/\mco(1))}   { H^1_{ {\ord}^\perp}  (F,{\ad \rho_f} \otimes_\mco E/\mco(1))}
\]
On the other hand, we have the canonical isomorphisms 
\[
\Phi_{\lambda_f}\cong H^1_f(F,\ad\rho_f\otimes E/\mco)^\vee\,.
\]  
It remains to note that by the balanced properties of Bloch-Kato Selmer groups and Lemma \ref{lem:Greenberg-Wiles}, the $\mco$-module above has finite length, equal to
\[ 
\length_\mco(H^1_f(F,\ad\rho_f\otimes E/\mco(1)))\,.
\] 

 (2): From Theorem~\ref{th:deformation-with-lambda} one gets an exact sequence
\[
0 \longrightarrow \bigwedge^d (\fp/\fp^2)^* \longrightarrow  \bigwedge^d(\fm/\fm^2)^* \longrightarrow   
	\frac{\Psi_{\lambda_f}}{ \Psi^{\ord}_{\lambda_f}} \longrightarrow  0. 
\]
and therefore
\begin{equation*}\label{defoPT2}
0 \longrightarrow    \frac  {\bigwedge^d (\fp/\fp^2)^*}{( z_f^\epsilon)} \longrightarrow   
		\frac{ \bigwedge^d(\fm/\fm^2)^*}{(\bigwedge^d \mathrm{res}_p)(z_f^\epsilon)}\longrightarrow   
			\frac{\Psi_{\lambda_f}}{ \Psi^{\ord}_{\lambda_f}} \longrightarrow  0\,. 
\end{equation*}
The desired isomorphism follows since after identification of $ \wedge^d(\fm/\fm^2)^*$ with $\mco$,  we have $ \bigwedge^d \mathrm{res}_p(z_f^\epsilon)=(\xi_f^\epsilon)$ which is the same as  the ideal $\eta_{\lambda_f}$ by construction.
\end{proof}

The next result is immediate from Proposition~\ref{pr:duality} and Theorem~\ref{th:regular}.

 \begin{corollary}
 With our previous assumptions, the following are equivalent
 	\begin{enumerate}[\quad\rm(1)]
 		\item  $\bbT^{\ord}$ is regular
 		\item  $H^1_{\ord^\perp}(F,\ad\bar\rho_f(1))=0$
 		\item $(z_f^\epsilon) = \bigwedge^d H^1_{\ord}(F,\ad\rho_f)$. \qed
 	\end{enumerate}
 \end{corollary}

\begin{remark} This corollary is an analog of  the well-known fact  that  the   $p$-part of the class group 
of $\bbQ(\zeta_p)^+$ being trivial  (Vandiver’s conjecture is that this should always be the case) is equivalent to  the group of cyclotomic units  having  index prime to $p$ inside the global units of $\bbQ(\zeta_p)$.
\end{remark}

We end  this section with the the  following  proposition whose proof we owe  to Gebhard Boeckle.  It should be compared to the statement in Proposition~\ref{pr:duality} for Selmer groups arising from  motives associated to adjoints of modular forms. Recall that the Leopoldt  conjecture state that if $F$ is a number field and  $F^{{\rm ab}, p}/F$ is the maximal abelian $p$ extension of $F$ unramified outside the primes above $p$ and $\infty$, then the $\bbZ_p$-rank of $ \Gal(F^{{\rm ab}, p}/F)$ is $r_2+1$. 

 \begin{proposition}
      Let $F$ be a number field and assume the Leopoldt Conjecture. One has an exact sequence 
      \[ 
      0 \to \frac{\prod_{v \in S_p} \mu_{p^\infty}(F_v)} {\mu_{p^\infty}(F)} \to A \to B \to 0\,,
      \]
      where $A\colonequals\tors(\Gal(F^{{\rm ab}, p}/F))$ and   $B\colonequals H^1_{S_p-\rm split}(F_{S_p}/F,\bbQ_p/\bbZ_p(1))$.
 \end{proposition}

\begin{proof}
Note that $A^*\colonequals{\rm cotor}(H^1(F_{S_p}/F,\bbQ_p/\bbZ_p))$.
Consider
\[
H^1(G_{F,S_p},\bbQ_p/\bbZ_p) = \Hom(G_{F,S_p}^{ab},\bbQ_p/\bbZ_p)=(G_{F,S_p}^{ab})^*=(\bbQ_p/\bbZ_p)^{r_2+1} \times  A^*\,.
\] 
By \cite[2.7.11]{Neukirch/Schmidt/Wingberg:2008}  the kernel of the map 
\[
H^1(G_{F,S_p},\bbQ_p/\bbZ_p) \to H^2(G_{F,S_p},\bbZ_p) 
\] 
is the divisible part of $H^1(G_{F,S_p},\bbQ_p/\bbZ_p)$, that is to say, $(\bbQ_p/\bbZ_p)^{r_2+1}$, and its image is the torsion subgroup of  $H^2(G_{F,S_p},\bbZ_p)$,  that is to say, all of $H^2(G_{F,S_p},\bbZ_p)$, under the Leopoldt conjecture. 

Further note that by Poitou-Tate one gets the isomorphism
\[
B=\Sha^1(G_{F,S_p},\bbQ_p/\bbZ_p(1))\cong \Sha^2(G_{F,S_p},\bbZ_p)^*\,.
\]
So there is a map $B^*  \to A^* $ that is the inclusion $\Sha^2(G_{F,S_p},\bbZ_p) \to H^2(G_{F,S_p},Z_p)$.
The cokernel is isomorphic to
the dual of the cokernel of 
$\mu_{p^\infty}(F)  \to \prod_{v \in S_p} \mu_{p^\infty}(F_v)$,  by   Poitout-Tate and local Tate duality.
\end{proof}

\section*{Acknowledgements}
This work is partly supported by National Science Foundation grants DMS-200985 (SBI) and DMS-2200390 (CBK), and  by a Simons Fellowship (CBK). The second and fourth authors thank TIFR, and the third author received funding from the European Research Council (ERC) under the European Union's Horizon 2020 research and innovation programme (grant agreement No. 884596). The paper has its origins in conversations between the second and fourth author in New York and Mumbai, during visits to  Columbia University  and  Tata Institute,  in which we tried to understand  more concretely the  results of  \cite{Iyengar/Khare/Manning:2022a} about congruence modules in  the codimension 1 case.  The second author would like to thank Gebhard Boeckle and Chris Skinner for helpful discussions.

\bibliographystyle{amsplain}

\end{document}